\documentclass[times]{oupau}

\usepackage{cmap}
\usepackage[linewidth=1pt]{mdframed}
\usepackage[all]{xy}

\newtheorem{notation}[theorem]{Notation}

\theoremstyle{definition}
\newtheorem{remark}[theorem]{Remark}
\newcommand{\Ann}{\mathop{\mathrm{Ann}}\nolimits}

\newcommand{\Gr}{\mathop{\mathrm{Gr}}\nolimits}

\newcommand{\LM}{\mathop{\mathsf{LM}}\nolimits}
\newcommand{\LT}{\mathop{\mathsf{LT}}\nolimits}
\newcommand{\N}{\mathop{\mathsf{N}}\nolimits}
\newcommand{\Qgr}{\mathop{\mathsf{Qgr}}\nolimits}
\newcommand{\qgr}{\mathop{\mathsf{qgr}}\nolimits}
\newcommand{\rk}{\mathop{\mathrm{rk}}\nolimits}
\newcommand{\Tors}{\mathop{\mathsf{Tors}}\nolimits}

\usepackage{amsmath,amssymb,amscd,mathrsfs}

\begin{document}

% Enter full title and short title for running headers
\title{Entropy of monomial algebras and derived categories}
\shorttitle{Entropy of monomial algebras and derived categories}

% Enter the publication year and the ID number of the paper
\volumeyear{????}
\paperID{??????}

% Author name(s)
\author{Lu Li\affil{1} and Dmitri Piontkovski\affil{2}}
% Abbreviated author name for running headers
\abbrevauthor{Lu Li and Dmitri Piontkovski}
% Abbreviated author name for first page header
\headabbrevauthor{Lu Li, and Dmitri Piontkovski}

\address{%
\affilnum{1}College of Mathematics and Statistics, Chongqing University,
Chongqing 401331, China
and
\affilnum{2}HSE University, Myasnitskaya ul. 20, Moscow 101000, Russia}

% Address / e-mail address of corresponding author
\correspdetails{dpiontkovski@hse.ru}

% Received/revised/accepted dates will be entered by the publisher during production of an accepted paper. Please do not edit these placeholders for submission.
\received{1 Month 20XX}
\revised{11 Month 20XX}
\accepted{21 Month 20XX}

% Enter details of editor communicating this article
\communicated{A. Editor}

\begin{abstract}
Let $A$ be a finitely presented associative monomial algebra.
We study the category $\mathsf{qgr}(A)$ which is a quotient of the category of graded finitely presented $A$-modules by the finite-dimensional ones. 
As this  category plays a role of the category of coherent sheaves on the corresponding noncommutative variety, we consider its bounded derived category $\mathbf{D}^b(\mathsf{qgr}(A))$.
We calculate  the categorical entropy 
of the Serre twist functor on $\mathbf{D}^b(\mathsf{qgr}(A))$ 
and show that it is equal to the (natural) logarithm of the entropy 
of the algebra $A$ itself. Moreover, we relate these two kinds of entropy 
with the topological entropy of the Ufnarovski graph of $A$ and the entropy of the
path algebra of the graph. If $A$ is a path algebra of some quiver, the categorical entropy is equal to the logarithm 
of the spectral radius of the quiver's  adjacency matrix.
\end{abstract}

\maketitle

\section{Introduction}

Let $A$ be a finitely presented monomial algebra
\begin{equation}
\label{form:monomial algebra}
    A=\frac{k\Gamma}{(F)},
\end{equation}
that is, 
a
quotient of a path algebra 
$k\Gamma$ 
for a finite quiver $\Gamma = (\Gamma_0, \Gamma_1)$ 
by an ideal generated by 
a fixed finite set $F$ of words (paths).

The (algebraic) entropy of a graded algebra $A$ is the exponential measure of its growth~\cite{Newman2000TheAlgebras, Piontkovski2000HilbertAlgebras}, 
\begin{equation*}
\mathsf{h}_{alg} (A) = \varlimsup_{n\to \infty} \sqrt[n]{\dim A_n},
\end{equation*}
where the graded component $A_n$ is the 
span of all paths of length $n$ in $A$. The logarithm $\log \mathsf{h}_{alg}$ of this entropy in the case of monomial algebra equals 
to the entropy $\mathsf{h}(L)$ of the (regular) language~\cite{Kuich1970OnLanguages} consisting of the nonzero 
paths in $A$, or to the topological entropy of the corresponding subshift \cite{Adler1965TopologicalEntropy}.  Moreover, $\mathsf{h}(L)$  is equal to the topological entropy $\mathsf{h}(Q_A)$ of the Ufnarovski graph $Q_A$ of $A$. 
In this connections, the entropy of the algebra  measures also the complexity both of the language $L$ and of the graph $Q_A$. 

On the other hand, one can consider $A$ as a coordinate ring of a noncommutative variety. 
We denote by $\mathsf{Gr}(A)$ and $\mathsf{gr}(A)$ the category of $\mathbb{Z}$-graded right modules and its subcategory of the finitely presented right modules, respectively.
Denote by $\mathsf{Tors}( A)$ and $\mathsf{tors}( A)$ their full 
subcategories %of $\mathsf{Gr}(A)$ and  the full subcategory of $\mathsf{gr}(A)$ consisting
of all torsion modules (which are sums of finite-dimensional ones).
Then the quotient category $\mathsf{Qgr}(A):= \mathsf{Gr}(A)/ \mathsf{Tors} (A)$ plays the role of the of the category of quasicoherent sheaves on the noncommutative variety defined by $A$~\cite{Artin1994NoncommutativeSchemes}. Moreover, as $A$ is coherent (see~\cite{Piontkovski1996GrobnerAlgebras} and Section~\ref{sec:monom_r_coherent} below),   
$\mathsf{gr}(A)$ is an abelian Serre subcategory, so that one can define the 
category of {\em coherent } sheaves as $\mathsf{qgr}(A):= \mathsf{gr}(A)/ \mathsf{tors} (A)$ \cite{Polishchuk2005NoncommutativeAlgebras}. 
Consider 
%Then the derived invariant of $A$ is 
the bounded derived category
$\mathbf{D}^b (\mathsf{qgr}(A))$ and the Serre twist functor on it. 

The notion of the entropy for (exact) endofunctors of triangulated categories (having a split generators)
has been introduced by Dimitrov, Haiden, Katzarkov, and Kontsevich~\cite{Dimitrov2014DynamicalCategories}. For some endomorphisms of projective varieties, the categorical entropy of the induced endofunctors of the derived category of coherent sheaves is connected with topological entropy
\cite{Kikuta2017OnCurves,Kikuta2019OnEntropy,Yoshioka2020CategoricalSurfaces}. 
The entropy and its value at zero is calculated for functors on the derived categories of algebraic varieties in~\cite{Dimitrov2014DynamicalCategories, Fan2018EntropyManifolds, Fan2018OnP-twists, Mattei2019CATEGORICALSURFACES, Kikuta2020ALINES, Ouchi2020ONTWISTS, Yoshioka2020CategoricalSurfaces}.  

Here we provide a calculation of the categorical entropy  for  the category $\mathbf{D}^b (\mathsf{qgr}(A))$ associated to the noncommutative variety defined by the coordinate ring $A$. In particular, we show that the entropy of the Serre twist functor is a constant and relate it to other type of entropy of topological and algebraic origin associated to the monomial algebra $A$.

For definitions and details related to the above versions of the entropy, see Section \ref{section: entropy}.
We will prove the following connection between them:
%Our first theorem asserts:
\begin{theorem}[Theorems \ref{theorem: The Holdaway--Smith homomorphism induces an equivalence of qgr} and~\ref{th: H_t = log h_alg}]
%[Proposition~\ref{prop:former_Main theorem-1},Theorems \ref{Main theorem-2}, and \ref{Main theorem-3} -?????????].
\label{th:main_intro}
The categorical entropy $\mathsf{h}_{t}(\mathbf{D}^b (\mathsf{qgr}(A)), \mathsf{S})$
of the Serre twist functor $\mathsf{S}$ 
 is a constant which is equal to  
\begin{equation*}
  \mathsf{h}_{t}(\mathbf{D}^b (\mathsf{qgr} A), \mathsf{S}) = 
%  \mathsf{h}_{top}(X_F) =  
%  \log \mathsf{h}_{alg}(kQ_A)   =  
  \log \mathsf{h}_{alg}(A).
\end{equation*}
\end{theorem}

In particular, we see that the algebraic entropy $\mathsf{h}_{alg}(A)$ is a derived invariant of the noncommutative variety associated to $A$.
It would be interesting to describe other noncommutative coherent algebras with this property. 
Note that the above entropy is related with the other kinds of entropy associated to $A$
as follows.
\begin{corollary}[Propositions~\ref{prop:former_Main theorem-1},~\ref{prop:Main theorem-3}]
\label{cor:intro other entropies}
For a monomial algebra $A$, we have the equalities
\begin{eqnarray*}
 \mathsf{h}_{t}(\mathbf{D}^b (\mathsf{qgr} A), \mathsf{S}) =
 \mathsf{h}_{top}(X_F) =  \mathsf{h}(L) = \log \mathsf{h}_{alg}(A) \ \ \\
  \ \  \ =  \mathsf{h}(Q_A) = 
 \log \mathsf{h}_{alg}(kQ_A) =
  \mathsf{h}_{t}(\mathbf{D}^b (\mathsf{qgr} (kQ_A)), \mathsf{S})=
 %\log_2 
 \log 
 \rho (Q_A), 
\end{eqnarray*}
where the entropies $\mathsf{h}, {\mathsf{h}}_{top}, {\mathsf{h}}_{alg}, \mathsf{h}_{t} $
relate the language of nonzero monomials of positive lenght in $A$, the subshift $X_F$ associated to $A$, the Ufnarovski graph $Q_A$
of $A$, the path algebra $kQ_A$, and the 
spectral radius $\rho (Q_A)$ of the adjacency matrix of the graph $Q_A$.
\end{corollary}

If $A = k\Gamma$ is the path algebra, then the equalities are degenerated to 
\begin{equation*}
 \mathsf{h}_{t}(\mathbf{D}^b (\mathsf{qgr} (k\Gamma)), \mathsf{S}) = 
 \log \mathsf{h}_{alg}(k \Gamma) = %\log_2 
 \log \rho (\Gamma),
\end{equation*}
since in this case $Q_A = \Gamma$.

Our approach is the following.
Holdaway and Smith~\cite{Holdaway2012AnQuivers} have established 
an equivalence of the category $\Qgr(A)$ with the analogous category 
$\Qgr(k Q_A)$ for the quiver algebra of the Ufnarovski graph $Q_A$ (see also the paper of Holdaway and Sisoda~\cite{Holdaway2014CategoryAlgebras} for the same result in a more general setting of non-connected algebras). 
This equivalence is induced by a map $f:A\to Q_A$ referred here as Holdaway--Smith map. We show that this map induces also an equaivalence 
of categories $\qgr A \simeq \qgr (k Q_A)$. Smith have described the category $\qgr$ for quiver algebras~\cite{Smith2012CategoryQuivers}. In particular, such category is semisimple. We use this fact to calculate the entropy 
$\mathsf{h}_{t}(\mathbf{D}^b (\mathsf{qgr}(A)), \mathsf{S}) = 
\mathsf{h}_{t}(\mathbf{D}^b (\mathsf{qgr}(k Q_A)), \mathsf{S})$. 

%The structure of the paper.....

In Section~\ref{section: entropy} we recall various definitions of entropy used in this paper. In particular, we give the definitions of the entropy of a graded algebra, the topological entropy of subshifts, formal languages, and directed graphs.
Then we immediately deduce from definitions that 
the entropies 
of 
a finitely presented  
monomial algebra $A$, of the (regular) language $L$ consisting of the nonzero monomials of positive length in $A$, and of the associated subshift $X_F$
%, the entropies 
are related by the equalities 
$$\mathsf{h}_{top}(X_F) =  \mathsf{h}(L) = 
  \log \mathsf{h}_{alg}(A)
$$
  (see Proposition~\ref{prop:former_Main theorem-1}). Then we follow~\cite{Dimitrov2014DynamicalCategories} to 
  define   the categorical entropy for the triangulated categories. 
  
  To describe the relations of this last entropy with the previous ones, 
we should define and describe the category $\qgr A$ for a finitely presented monomial algebra $A$. First, we recall the fact the $A$ is coherent in Section~\ref{sec:monom_r_coherent}. A proof~\cite{Piontkovski1996GrobnerAlgebras} of this fact is given in a less general situation    than the one considered here (that is, for the connected monomial algebras in place of the quotients of quiver algebras). Here we give another self-contained proof.  We call an algebra with a fixed multiplicative ordering on the monomials in generators  {\em right Groebner finite} 
if it has finite Groebner basis of relations and each right-sided ideal in it has finite Groebner basis as well (Definition~\ref{def:Groebner_finite}). Then we use
\begin{proposition}[Corollary~\ref{cor:groebner_finite_are_coherent}]
\label{prop:intro:GF}
Each right Groebner finite algebra is right coherent.
\end{proposition}
As each finitely presented 
monomial algebra is right Groebner finite (Lemma~\ref{lem:finite_GB}), we deduce that 
such an algebra is coherent (Corollary~\ref{corollary: coherent algebras}). 
Note that $A$ is coherent in the general non-graded sense.

Since $A$ is coherent, the category $\qgr A$ exists. 
Moreover, we have 
\begin{proposition}[Proposition~\ref{prop:qgr=fp Qgr}]
\label{prop:intro:qgr = fp Qgr}
Let $R$ be a 
finitely generated 
right graded coherent algebra.  Then there is a natural 
%   inclusion functor $\mathsf{gr} (R) \to \mathsf{Gr} (R)$
%induces an 
equivalence of categories $\mathsf{qgr} R \equiv \mathsf{fp} (\mathsf{Qgr} R)$,
where $ \mathsf{fp} $ denotes the subcategory of finitely presented objects.
\end{proposition}
This proposition looks standard, however, we have not found an explicit proof of it in the literature. 
We apply the theory of locally coherent Grothendieck categories~\cite{Krause1997TheCategory} 
to prove it in Section~\ref{App:locally_coherent_categories}.

To continue to study the category $\qgr A$,
we recall in the next Section~\ref{sec:Ufnarovski_n_Holdaway--Smith} the definitions of the Ufnarovski graph $Q_A$ and the Holdaway-Smith map of algebras $f: A\to k Q_A$. 
Then we show in Theorem~\ref{theorem: The Holdaway--Smith homomorphism induces an equivalence of qgr} that this homomorphism induces an equivalence 
 $$
 \mathsf{qgr} (A) \equiv \mathsf{qgr} (kQ_A).
 $$
In the next Section~\ref{sec:main_theorem}, we use some properties 
of the category $\qgr (kQ_A)$ established by Smith~\cite{Smith2012CategoryQuivers} (mainly, the semi-simplicity) to establish Theorem~\ref{th:main_intro}.
The key lemma is the following. For the definition of complexity used in it, we refer the reader to~\cite{Dimitrov2014DynamicalCategories} or to Subsection~\ref{subs:category_ent_def}.
\begin{lemma}[Lemma~\ref{lemma: rk=delta}]
\label{lemma: intro: rk=delta}
Let  $\mathcal{O}$ be a generator and  $X$ be an object of a semi-simple abelian category $\mathsf{C}$. 
Then the complexity $\delta_t(\mathcal{O}, X) $ of $X$ relative to $\mathcal{O}$ in the category $\mathbf{D}^b(\mathsf{C})$
is a constant which is equal to $\rk_{\mathcal{O}}(X)$, that is, 
the minimal number $s$ such that $X$ is a direct summand of $\mathcal{O}^s$ in 
$\mathbf{D}^b(\mathsf{C})$.
\end{lemma}
Since the categorical entropy is defined in terms of complexity,
we use the lemma to deduce in Theorem~\ref{Main theorem-2} that 
$\mathsf{h}_{t}(\mathbf{D}^b (\mathsf{qgr}(kQ_A)), \mathsf S) = 
 % \mathrm{lim}_{m\rightarrow +\infty} \frac{1}{m} \log N_m = 
  \log  \mathsf{h}_{alg}(kQ_A)$.
  In the view of the above equivalence  $
 \mathsf{qgr} (A) \equiv \mathsf{qgr} (kQ_A),
 $
Theorem~\ref{th:main_intro} now follows. 
%The above equivalence give also the connections of the entropies of the algebra $A$ and the corresponding entropies of the path algebra $k Q_A$.
In Section~\ref{section: An example}, we consider a particular example. 

 %\subsection*{Acknowledgement}

% An acknowledgements section is started with \verb"\ack" or \verb"\acks"
% for \textit{Acknowledgement} or \textit{Acknowledgements}, respectively. It
% must be placed just before the references (or before the appendix when applicable).

\section{Algebraic entropy, topological entropy and category-theoretical entropy}
\label{section: entropy}

\subsection{Algebraic entropy of a graded algebra}
\label{subsection: Algebric entropy}
Let $A = \bigoplus_{m\geq 0}A_m$ be a graded algebra over a field $k$. Assume that each graded component $A_m$ is finite-dimensional as a vector space over $k$, so that $a_m = \mathsf{dim}_k A_m < +\infty$. Newman, Schneider and Shalev \cite{Newman2000TheAlgebras} defined the \textsf{algebric entropy} of $A$ by
\begin{equation*}
    \mathsf{h}_{alg}(A):= \varlimsup_{m\rightarrow +\infty} \sqrt[m]{a_m}.
\end{equation*}
In other words, it is the exponent of growth of the algebra $A$~\cite{Piontkovski2000HilbertAlgebras}.

Note that if the algebra $A$ is generated in degrees 0 and 1 (like the algebras considered in this paper), then the sequence $\{a_m \}$ is sub-multiplicative:
$a_{p} a_{q} \ge a_{p+q}$. By Fekete's Lemma, in this case the above limit exists:
\begin{equation*}
    \mathsf{h}_{alg}(A)= \lim_{m\rightarrow \infty} \sqrt[m]{a_m} = \inf_{m\ge 0} \sqrt[m]{a_m}.
\end{equation*}
Note that the radius of convergence of the Hilbert series $H_A(z) = \sum_m a_m z^m$ is  $ r = 1/\mathsf{h}_{alg}(A)$.
%By the d'Alembert formula for the convergence radius, for (infinite-dimensional) algebras generated in degrees 0 and 1 we have 
%\begin{equation*}
%    \mathsf{h}_{alg}(A)= \lim_{m\rightarrow \infty} a_{m+1}/a_m.
    % Should be replaced by:
% \varliminf_{m\rightarrow +\infty} a_{m+1}/a_m  \le \mathsf{h}_{alg}(A) \le \varlimsup_{m\rightarrow +\infty} a_{m+1}/a_m
%\end{equation*}

\subsection{Topological entropy of formal languages and subshifts}

\label{subs:top_entrop_fromal_langs}

Topological entropy was first introduced in 1965 by Adler, Konheim and McAndrew  in \cite{Adler1965TopologicalEntropy}. 
For a compact space $X$, let $\mathcal U$ be an open cover of $X$. 
The \textsf{entropy} of $\mathcal U$ is
\begin{equation*}
 \mathsf  H(\mathcal U) = \log N(\mathcal U),
\end{equation*}
where $N(\mathcal U) = \mathrm{min}\{|\mathcal V| : \mathcal V \text{ is a finite subcover of } \mathcal U \}$. Here $\log$ denotes the logarithm to some fixed base, see Remark~\ref{rem:log_base} below. For $m \in \mathbb Z_{>0}$ and open covers $\mathcal U_1, \cdots, \mathcal U_m$  let
\begin{equation*}
\mathcal U_1 \vee \cdots \vee \mathcal U_m = \{\bigcap_{i=1}^m U_i : U_i \in \mathcal U_i\} .   
\end{equation*}
Let  $f: X \rightarrow X$ be a continuous self-map.
The \textsf{topological entropy} of $f$ with respect to $\mathcal U$ is
\begin{equation*}
   \mathsf  H_{top}(f,\mathcal U)= \lim_{n\rightarrow +\infty} \frac{ \mathsf H(\mathcal U \vee f^{-1}(\mathcal U ) \vee \cdots \vee f^{-n+1}(\mathcal U ) )}{n}.
\end{equation*}
The \textsf{topological entropy} of $f$ is
\begin{equation*}
   \mathsf h_{top}(f):= \sup \{ \mathsf H_{top}(f, \mathcal U) : \mathcal U \text{ is an open cover of } X  \}.
\end{equation*}

We will consider the entropy of a subshift, which is a special case of the topological entropy.

Let $G=\{x_1, x_2,\cdots, x_n\}$ be an alphabet, $G^{+}$ the set of words in $G$ with finite length and $G^* := G^{+} \cup \{\epsilon\}$, where $\epsilon$ denotes the empty word. 
We write the usual concatenation of words as $uv$ for $u, v \in G^{+}$, and $v \lhd u$ if $v$ is a subword of $u$. Under word concatenation, $G^{*}$ forms a free monoid generated by $G$ with identity given by the empty word $\epsilon$. 
All subsets of $G^{\mathbb N}$ for a finite alphabet $G$ will automatically adopt the subspace topology, where $G$ is a discrete space and $G^{\mathbb N}$ is endowed with the product topology.

%The \textsf{full} $G$-\textsf{shift} is the dynamical system consisting of the set of infinite symbol sequences, together with the shift map $\sigma$ that shifts all coordinates to the left. More formally, our topological space is
%\begin{equation*}
%\Sigma_n:=G^{\mathbb Z_+} = \{a = (a_0,a_1,\cdots, a_k, \cdots) | a_i \in G, \text{for all} \ i\in \mathbb{Z}_+ \}
%\end{equation*}
%endowed with the product topology of the discrete topology on $G$
%, and the self-map  $\sigma: \Sigma_n \rightarrow \Sigma_n$ is
%\begin{equation*}
%(a_0,a_1,\cdots, a_k, \cdots)\mapsto (a_1,a_2,\cdots, a_k, \cdots)
%\end{equation*}

Let $F$ be a finite subset of $G^+$ which we call a set of forbidden words. The shift of finite
type $X_F$ associated to $F$ is the space
\begin{equation*}
   X_F = \{ x \in G^{\mathbb N} | u \lhd x \Rightarrow u \not\in F \} . 
\end{equation*}
Words in $X_F$ are called \textsf{legal}.  We write $L_n $ for the set of legal words of length $n$.
The \textsf{topological entropy} of  $X_F$ is 
\begin{equation*}
  \mathsf{h}_{top}(X_F):=  \varlimsup_{n\rightarrow +\infty} \frac{1}{n} \log|L_n|.
\end{equation*}

Note that $X_F$ is a closed subspace of $G^{\mathbb N}$ and the 
the self-map  $\sigma: X_F \rightarrow X_F$ defined as 
\begin{equation*}
\sigma: (a_0,a_1,\cdots, a_k, \cdots)\mapsto (a_1,a_2,\cdots, a_k, \cdots)
\end{equation*}
is continuous. Then $\mathsf{h}_{top}(X_F) = \mathsf h_{top}(\sigma)$. We are grateful to one of the anonymous reviewers who put our attention to this remark.

The  formal language $L = \bigcup_n L_n$ is the set of all nonzero paths of positive length in the monomial algebra $A$. Its entropy is defined to be 
\begin{equation*}
\mathsf{h} (L) = \log \varlimsup_{n\rightarrow +\infty} \sqrt[n]{|L_n|},
\end{equation*}
see~\cite{Kuich1970OnLanguages}. 

\begin{proposition}
\label{prop:former_Main theorem-1}
Let $A$ be a graded monomial algebra of the form
\begin{equation*}
A = k\Gamma/(F),
%\frac{k\langle G\rangle}{(F)},
\end{equation*}
where $\Gamma =(\Gamma_0, \Gamma_1)$ is a finite quiver with
$\Gamma_1 = \{x_1,x_2,\cdots, x_n\}$, 
and $(F)$ denotes the ideal generated by 
a finite set
$F=\{w_1,\cdots,w_s\}$ is 
 of words in the alphabet $\Gamma_1$, . 
We denote  by $X_F$ 
the shift of finite
type  associated to $F$, and by $L$
the formal language as above. Then
\begin{equation*}
\mathsf{h}_{top}(X_F) = \mathsf{h} (L) = 
  \log \mathsf{h}_{alg}(A).
\end{equation*}
\end{proposition}

\begin{proof}
We have
\begin{equation*}
    \mathsf{h}_{alg}(A):= \varlimsup_{n\rightarrow +\infty}\sqrt[n]{\mathsf{dim}_k A_n} = \varlimsup_{n\rightarrow +\infty} \sqrt[n]{\mathsf{dim}_k |L_n|}.
\end{equation*}
Hence
\begin{equation*}
\mathsf{h}_{top}(X_F) =  \varlimsup_{n\rightarrow +\infty} \frac{1}{n} \log|L_n| = \varlimsup_{n\rightarrow +\infty} \log \sqrt[n]{\mathsf{dim}_k L_n} = \mathsf{h} (L)  = \log \mathsf{h}_{alg}(A).   \ \ \square
\end{equation*}
\end{proof}

\subsection{Graph entropy}

\label{subs:graph_entropy}

Let $G = (G_0, G_1)$ be a directed graph. 
Suppose that $G$ is finite, and let $A_G$ denotes the adjacency matrix of $G$. Then the topological entropy of the graph $\mathsf{h}(G)$ is defined to be the topological entropy $\mathsf{h}_{top}(X_F)$
of the subshift $X_{G_1}$ of the paths in $G$, that is, 
\begin{equation*}
\mathsf{h}(G) := \varlimsup_{n\rightarrow +\infty}\frac{\log a_n}{n},
\end{equation*}
where $a_n$
the number of paths of length $n$. It is well-known that 
\begin{equation*}
\mathsf{h}(G) = \log r(A_G),
\end{equation*}
where  $r(A_G)$ is the spectral radius of $A_G$.
%Indeed, the generating function $H_G(z)$ of the number of paths $a_n$ of the graph $G$ is equal to $\det(I+z(I-A_G))/\det(I-zA_G)$, ...
Since the entropy of the subshift $X_{G_1}$ is equal to the algebraic entropy of the path algebra $kG$, we have the equality
\begin{equation*}
\mathsf{h}(G) =  \log \mathsf{h}_{alg}(kG).
\end{equation*}

Note that there are various definitions of entropy for infinite graphs (such as the topological entropy, loop entropy, block entropy, and other). 
There is also another version of the entropy defined in terms of the associated $C^*$-algebra. Still, in the case of finite graphs, all these  entropies give the same value $h(G) = \log r(A_G)$, see~\cite{Jeong2006TopologicalC-algebras} and references therein.

\subsection{Category-theoretical entropy}

\label{subs:category_ent_def}

For triangulated categories and thick subcategories we refer to  
\cite{Weibel2013AnAlgebra} and
\cite{Neeman2014TriangulatedCategories.AM-148}.

\begin{definition}
Let $\mathsf{C}$ be a triangulated category.  We say a full triangulated subcategory $\mathsf L$ is \textsf{thick} if it is closed under taking direct summands.
\end{definition}

\begin{definition}
An object $\mathcal{O}$ of $\mathsf C$ is called a  \textsf{generator} if the smallest thick subcategory of $\mathsf C$ containing $\mathcal{O}$ is equal to $\mathsf C$ itself . 
\end{definition}

\begin{definition}(\cite{Dimitrov2014DynamicalCategories}, Definition 2.1)
\label{def:complexity}
Let $\mathsf{C}$ be a triangulated category with a generator $\mathcal O$. Let $E$ be a object of a triangulated category $\mathsf T$. 
There is an object $E'$ and a tower of distinguished triangles
\begin{eqnarray}
\label{s}
\xymatrix@=.4cm{
  E_0 \ar[rr]^{} & &  E_1 \ar[rr]^{} \ar[ld]^{} & &E_2 \ar[r]^{} \ar[ld]^{}&\cdots \ar[r]^{}& E_{p-1}\ar[rr]^{} & &E_p\cong E\oplus E'  \ar[ld]^{}  \\
  &  \mathcal{O}[n_1] \ar@{-->}[lu]_{} & & \mathcal{O}[n_2]  \ar@{-->}[lu]_{}& & \cdots  & & \mathcal{O}[n_p] \ar@{-->}[lu]_{}
   }
    \end{eqnarray} 
with $E_0 = 0$, $p\geq 0$, and $n_i\in \mathbb{Z}$.
Let $t$ be a real number. To each tower of distinguished triangles of the form~(\ref{s})  
we associate the exponential
%we associate the exponential 
sum $\sum_{i=1}^p e^{n_it}$.
Let $S_t\subset \mathbb{R}$ be the set of all such sums for a given $t$. The \textsf{complexity} of $E$ with respect
to $\mathcal{O}$ is the function $\delta_t(\mathcal{O}, E):\mathbb{R}\rightarrow [0,+\infty]$
of $t$, given by $\delta_t(\mathcal{O}, E)= \mathrm{inf} \ S_t$.
\end{definition}

\begin{lemma}(Subadditivity)
\label{lemma: Subadditivity}
\begin{equation*}
\delta_t(\mathcal{O}, E_1 \oplus E_2) \leq \delta_t(\mathcal{O}, E_1) + \delta_t(\mathcal{O}, E_2).
\end{equation*}
\end{lemma}

\begin{definition}
\label{def:cat_entropy}
Let $F : \mathsf C \rightarrow \mathsf C$ be an exact endofunctor of a triangulated category $\mathsf C$ with generator $\mathcal O$. The \textsf{entropy} of $F$ is the function $\mathsf h_t(F) : R \rightarrow [-\infty,+\infty)$ of $t$ given by
\begin{equation*}
\mathsf{h}_t(\mathsf C, F)= \mathrm{lim} \frac{1}{n} \log \delta_t(\mathcal{O}, F^n\mathcal{O}).
\end{equation*}
\end{definition}

It is shown in
~\cite[Lemma~2.5]{Dimitrov2014DynamicalCategories}
that $\mathsf{h}_t(\mathsf C, F)$ is well-defined, i.e., the limit defining $\mathsf{h}_t(\mathsf C, F)$ exists and is independent of the choice of generator $\mathcal O$.

\begin{remark}
\label{rem:log_base}
In
the above definitions of topological entropy of formal languages and subshift in Subsection~\ref{subs:top_entrop_fromal_langs} and of graph entropy in Subsection~\ref{subs:graph_entropy}, the  $\log$ symbol denotes the logarithm to an arbitrary  fixed base $a>1$.
In contrast, in Definition~\ref{def:cat_entropy} due to~\cite{Dimitrov2014DynamicalCategories} the symbol $\log$ denotes the natural logarithm. We are grateful to an anonymous referee who have pointed out this fact. If one needs to use $\log_a$ in place of the natural logarithm here, one should replace the exponents $e^{n_it}$ by $a^{n_it}$ in Definition~\ref{def:complexity} of complexity. 
In the view of this remark, below we assume that all our logarithms are to the same base.
\end{remark}

\section{Finitely presented monomial algebras are coherent}
\label{sec:monom_r_coherent}

It is proved in \cite{Piontkovski1996GrobnerAlgebras}
that each connected finitely presented monomial algebra 
is coherent. 
%Moreover, it follows from~\cite[Corollary~3 and Theorem~8]{piontkovskii2001non}
%that each (non necessary connected) finitely presented monomial algebra is coherent.
Here we give a self-contained proof of this fact
in the more general setting of monomial quotients of path algebras.
In particular, we give a new proof of key Lemma~\ref{lem:finite_GB} which establishes that each finitely presented monomial algebra is right Groebner finite, that is, each finitely generated right ideal in it admits finite Groebner basis.

Let us firstly recall standard facts from noncommutative Groebner bases theory. Note that the noncommutative 
Groebner bases are also known as Groebner--Shirshov bases. We use a version of the theory suitable for quiver algebras, see \cite{Farkas1993SynergyAlgebras}, \cite{Green2000MultiplicativeBases}. Our terminology is closed to the one from~\cite{Ufnarovskij1995CombinatorialAlgebra}.

In this section, let $A = k\Gamma /I$ be a
quotient of a path algebra 
$k\Gamma$ 
for a finite quiver $\Gamma = (\Gamma_0, \Gamma_1)$. Initially, we do not assume that $A$ is monomial.
The paths of $\Gamma$ and the unit are called {\em monomials};
the monomials form a multiplicative submonoid of $k \Gamma$.
Let $\leq $ denote an arbitrary multiplicative well-order on the
monomials 
such that 
%the empty path 
$1$ is the minimal element. 
An example of such order is the degree-lexicographical one 
(then $m_1 < m_2$ iff either $\deg m_1 < \deg m_2$ or $\deg m_1 = \deg m_2$ and $m_1$ is less then $m_2$ lexicographically). 
%We assume also that each zero-length path $x\in \Gamma_0$ is 
%less than any path of positive length.
Then a monomial $m\in k\Gamma$ is called \textsf{normal} if its image $m+I$ in $A$ cannot be presented as a linear combination of strictly less monomials. Then each element of $f\in A$ can uniquely be presented as $f = \N(f) +I$, 
where the \textsf {normal form} $\N(f) \in k\Gamma$ is a linear combination of normal monomials $\sum_t \alpha_t m_t$,
where for nonzero $f\in A$ one can assume that $\alpha_t \in k^\times$ and $m_t$ are normal monomial.
Here the largest monomial $m_t$ is called the \textsf{leading monomial} of $f$, and the corresponding  summand $\alpha_t m_t$  is called the  \textsf{leading term} of $f$. They are denoted by $\LM(f)$ and $\LT(f)$, respectively.
The element $\alpha_t \in k$ here is called the  \textsf{leading coefficient} of $f$.

We identify each element $f\in A$ with its normal form $\N(f) \in k\Gamma$.
So, we extend the map $N: f\mapsto N(f)$ to a $k$-linear projector $\N:k\Gamma \to k\Gamma$.
The multiplication of (the normal forms of) elements of the quotient algebra $A$ is denoted by $*$, so that we put
\begin{equation*}
a*b = \N(ab)
\end{equation*}
for all $a,b\in k\Gamma$.

A subset $G\subset I$ of the two-sided ideal $I$ is called a \textsf{Groebner basis} of $I$, if for each $f\in I$ there exist two monomials $p,q\in k\Gamma$ such that $\LM (f) = p \LM(g) q$ for some $g\in G$. A Groebner basis is always a generating set of the ideal $I$. In particular, if $I$ is generated by a set $M = \{m_1, m_2, \dots \}$ of monomials, then $M$ is a Groebner basis of $I$.

Similarly, a subset $H$ of a two-sided (respectively, right-sided) ideal $J$ of the quotient algebra $A = k\Gamma/I$ is called a Groebner basis of $J$ if for each $f\in J$ there are normal monomials $p, q$ (resp., a single normal monomial $q$) such that $\LM (f) = p \LM(g) q$  (resp., $\LM (f) = \LM(g) q$). In the case of right ideal, a Groebner basis of the ideal $J$ is also a Groebner basis of it as a right submodule of $A$ in the sense of~\cite[Prop.~4.2]{Green2000MultiplicativeBases}.

A version of the next Lemma for connected monomial algebras is given in \cite[Th.~1]{Piontkovski1996GrobnerAlgebras}.
\begin{lemma}
\label{lem:finite_GB}
Suppose that the ideal $I$ of relations of the algebra $A = k\Gamma/I$ is generated by 
a finite set of monomials. Suppose that $J$ is a finitely generated right-sided ideal in $A$. Then $J$ admits a finite Groebner basis (with respect to some monomial ordering). 

More precisely, if the defining monomial relations of $A$ have degrees at most $l+1$ and the generators of $J$ have degrees at most $d$, then for each monomial ordering which refines the partial order by the degrees of monomials (for example, the   degree-lexicographical order) there exists a Groebner basis of $J$   such that all its elements have degrees at most $d+l$. The elements of this basis have the form $hm$, where $\deg h \le d$ and $m$ is a normal monomial.
\end{lemma}
\begin{proof} 
Let $M$ denotes the minimal set of the monomial generators of $I$, and let $F = \{f_1, \dots , f_s\}$ be the minimal generating set of $J$ such that $\deg f_i \le d$ for all $i$. 
Consider the set $G$ consisting of all elements of $J$ having the form $g = hm$, where $\deg h \le d$ and 
$m$ is a normal monomial (maybe, empty) such that $\deg h +\deg m \le d+l$. Recall that the elements of $A$ (in particular, the element $g$) are identified with their normal forms, so that the equality $g=hm$ for normal forms implies
that the elements $h$ and $m$ of $k Q$ are normal and $hm = h *m$. 

Then the set $G$ has the following properties.

(a) $G$ is a generating set of $J$, since $F\subset G$, 

(b) If $g= hm$ as above such that $\deg m \ge l$, then for each monomial $q$ we have either $g*q=0$
or $g*q = gq$.

(c) Moreover, for each $g= hm$ as above with no additional restrictions on $\deg m$ and each normal monomial $q$
we have either $g*q = 0$ or $g*q = g'q'$, where $g'\in G$ and $q'$ is some right divisor of $q$. 

Indeed, assume that $g*q \ne 0$. If $q=bc$
and $\deg m + \deg b \ge l$, then we have $g*q = (g*b) c = g'c$, where  $g' = g*b = h*(mb)$ belongs to $G$. Otherwise,
$g' = g*q \in G$.

(d) If $g_1, g_2 \in G$ and $\LM (g_1) = \LM (g_2) m$ for some monomial $m$, then the element $g = g_1 +\alpha g_2 * m$
belongs to $G$ for each $\alpha \in k.$

(e) $G$ is a Groebner basis of $J$.

Indeed, let $f$ be an arbitrary element of $J$. We should prove that $\LM (f)$ is left divisible by some $\LM(g)$ with $g\in G$. 
By (a) and (c), $f$ can be presented as a sum
\begin{equation}
\label{eq:pres_f}
f = \sum_{i=1}^N g_i p_i, 
\end{equation}
where $g_i \in G$, the elements $p_i$ are normal, and $g_i * p_i = g_i p_i$. 
Let the monomial $\max_i \LM(g_i p_i)$ is called the leading monomial of the presentation. Among all such representations for $f$, we fix the one with the least possible leading monomial.

Let $\LM(g_1 p_1) \ge \LM(g_2 p_2) \ge \dots \ge \LM(g_N p_N)$ are listed in decreasing order. If $\LM(f) = \LM(g_1 p_1) = \LM(g_1) \LM(p_1)$,
then $\LM(f)$ is left divisible by $\LM(g_1)$. Otherwise, there is $n$ such that 
$\LM(g_1 p_1) =\dots = \LM(g_n p_n) > \LM(g_{n+1} p_{n+1})\ge \dots$ 
and $\LT(g_1 p_1) +\dots +\LT(g_n p_n) =0$.
We can assume also that $
\LM(g_1) \ge \dots \ge \LM(g_n)$. 
For $i\in[2,n]$, we have in addition $\LM(g_1) \LM(p_1) = \LM(g_i) \LM(p_i)$, 
so that there is a normal monomial $m_i$ such that $\LM(g_1) = \LM(g_i) m_i$ and 
$\LM(p_i) = m_i \LM(p_1) $. For completeness, we put also $m_1=1$.

Let $\alpha_i$ denotes 
%$\alpha_1 , \dots , \alpha_n $ 
the leading coefficients of $p_i$. 

Consider the element 
$$
s = \sum_{i=1}^n g_i \LT(p_i) = \sum_{i=1}^n \alpha_i g_i m_i \LM (p_1) = g \LM (p_1),
$$
where the element $g = \sum_{i=1}^n \alpha_i g_i m_i$ belongs to $G$ by (d).
Since the sum of the leading terms in the above presentation of $s$ is zero, the leading monomial
$\LM (s) = \LM (g) \LM (p_1)$ is less then the leading monomial 
$\LM (g_1) \LM (p_1)$ of the presentation~(\ref{eq:pres_f}) for $f$. 

Then we get another presentation for $f$,
$$
f = \sum_{i=1}^n g_i (p_i-\LT (p_i)) + g \LM (p_1) +  \sum_{i=n+1}^N g_i p_i
$$
such that its leading monomial is less than the one of the presentation~(\ref{eq:pres_f}). A contradiction.

Now, note that the set of all possible leading monomials of the elements of $G$ is finite. Let $G'$ 
be any finite subset of $G$ such that for each $g\in G$, there is an element $g'\in G'$ such that $\LM g$
is left divisible by $\LM g'$. Then the set $G' \cup F$ is a finite Groebner basis of $J$ which satisfies the conditions of Lemma. % $\square$
\end{proof}

\begin{proposition}
\label{prop:Groebner_f_p_ideal}
%Fix a degree-lexicographical ordering on the paths of $\Gamma$.
Let $A = k\Gamma/I$ a quotient algebra of a path algebra $k\Gamma$, 
where the ideal $I$ in $k\Gamma$ has finite Groebner basis.
Assume that a right-sided  ideal $J$ in the algebra $A$ also has finite Groebner basis.
Then $J$ is finitely presented as a right $A$-module.
\end{proposition}

\begin{proof}
Let $S = \{s_1, \dots , s_m\}$ and $G =\{g_1, \dots, g_s\}$ be minimal Groebner bases for $I$ and $J$, respectively. 
Let $D$ be the set of all proper 
right
%left 
divisors of the leading terms of the elements of $S$.
Let $\hat A$ be the associated monomial algebra $\hat A = k \Gamma /(\LM(S))$. 
Then  the right annihilator $\Ann_{\hat A} w$ in $\hat A$ of a normal  monomial $w$ is generated by some elements of $D$.

Suppose that $m\in \Ann_{\hat A} \LM(g_i)$ for some $g_i \in G$. Then the leading monomial of the element $g_i* m$ of $J$
is less than $g_i m$. This leading monomial is left divisible by some $\LM (g_{i_1})$; say, this leading monomial is $\LM (g_{i_1}) m_1$. Next, the leading monomial of the element 
$$
  g_i* m - g_{i_1} * m_1
$$
is left divisible by some $\LM(g_{i_2})$, etc. After a finite number $N$ of steps, we get zero.
So, we obtain a presentation
\begin{equation}
    \label{eq:pres_4_gi}
g_i *m = \sum_{j=1}^N g_{i_j} *m_j,
\end{equation}
where each for each $j$ we have  $\LM(g_{i_j} m_j) < g_i m$.

Since $J$ is generated by $s$ elements, we have a short exact sequence 
$$
0\to \Omega \to A^s \stackrel{\pi}\longrightarrow  J \to 0
$$
for some submodule $\Omega \subset A^s$. We should prove that $\Omega$ is finitely generated.
Let $\tilde g_1, \dots , \tilde g_s$ be the free generators of the above free module $P = A^s$.
We claim that $\Omega$ is generated by the elements 
$$
S(i,m) = \tilde g_i m - \sum_{j=1}^N \tilde g_{i_j} m_j,
$$
where $g_i$ runs over $G$, $m$ runs over the monomial generators of $\Ann_{\hat A} \LM(g_i)$, and the sum arises from~(\ref{eq:pres_4_gi}).
Note that the above map $ \pi: A^s \to J$ sends $\tilde g_i$ to $g_i$, so that (\ref{eq:pres_4_gi}) ensures that each
$S(i,m)$ belongs to $\Omega$. 

To prove this claim,  let us extend the 
%degree-lexicographical 
monomial order to the free module $P$ via the $k$-linear injection $P\to k\Gamma$ which sends  $\tilde g_i$ to 
 $\LM (g_i)$. It is indeed an injection since $G$ is a minimal Groebner basis, so that the monomials $\LM (g_i)$ are not left divisible by each other. Each element of $\Omega$ has the form
 $$
 \omega = \sum_t \tilde g_t b_t,
 $$
 where $\sum_t g_t * b_t = 0$. 
 Assume that $\omega $ does not belong to the submodule $\Omega_S$ generated by the elements $ S(i,m) $  
 and that 
 its leading monomial is the least one among all elements having this property.

The leading term of $\omega$ is equal to $\tilde g_i \LT(b_i)$ for some $t=i$. Let $l+1$ be the maximal length the monomials in $\LM(S)$.
Since $\pi (\omega) = 0$, we have  $\LT(b_i) \in \Ann_{\hat A} \LM(g_i)$, so that $\LT b_i = m q$ for some normal monomial $m$
of length at most $l$ and some normal $q$. Using~(\ref{eq:pres_4_gi}), we obtain another element 
$$
\omega' = \omega - S(i,m) q  =  \sum_{t\ne i} \tilde g_t b_t + \tilde g_i (b_i-\LT(b_i)) + \sum_{j=1}^N \tilde g_{i_j} (m_j *q), 
$$
which belongs to $\Omega \setminus \Omega_S$ and has less leading monomial than $\omega$. This contradiction shows that 
$\Omega= \Omega_S$.
\end{proof}

\begin{definition}
\label{def:Groebner_finite}
Fix an admissible monomial ordering on the paths of a quiver $\Gamma$.
We call an algebra $A = k\Gamma/I$ \textsf{right Groebner finite}
if the two-sided ideal $I$ of relations has finite Groebner 
basis, and each right-sided ideal in $A$ has finite Groebner basis as well.
\end{definition}
% We call an algebra $A$ \textsf{right Groebner finite} if each right ideal in $A$ has finite Groebner basis.
Lemma~\ref{lem:finite_GB} shows that each finitely presented monomial algebra is right Groebner finite. A more general class of Groebner finite algebras
%with finite Groebner basis of relations 
has been considered in~\cite{Piontkovski2001Non-commutativeSemigroups}.

Similar notions of Groebner finite algebras (for two-sided ideals) and Groebner coherent rings (in the commutative settings) have been under consideration in~\cite{Leamer2006GrobnerAlgebras} and~\cite{Nagpal2020Grobner-coherentModules}. 

For a more restrictive property of universal coherence for a class of Groebner finite algebras, see~\cite[Prop.~4.10]{Piontkovski2005LinearRings}.

Proposition~\ref{prop:Groebner_f_p_ideal} implies

\begin{corollary}
\label{cor:groebner_finite_are_coherent}
Each right Groebner finite algebra is right coherent.  
\end{corollary}

In the view of Lemma~\ref{lem:finite_GB}, we get

\begin{corollary}
\label{corollary: coherent algebras}
Suppose that the ideal $I$ of the algebra $ k\Gamma$ is generated by 
a finite set of monomials. Then the quotient algebra $A = k\Gamma/I$ is right and left coherent.
\end{corollary}

\section{Locally coherent Grothendieck categories of graded modules}

In this 
%appendix, 
section, 
we collect some known facts about locally coherent Grothendieck categories. 
We use these facts to connect the properties of categories $\qgr$ and $\Qgr$ in Propostition~\ref{prop:qgr=fp Qgr}.

%about locally coherent Grothendieck categories

\label{App:locally_coherent_categories}

 Let $\mathcal{A}$ be an abelian category, and let $\mathcal{S}$ be a nonempty full subcategory of $\mathcal{A}$. $\mathcal{S}$ is a \textsf{Serre subcategory} provided that it is closed under extensions, subobjects, and quotient objects. 
 For Serre subcategories and quotient categories we refer to  \cite{Gabriel1962DesAbeliennes},
\cite{Grothendieck1957SurI},
\cite{MacLane1963NaturalCommutativity}, and
\cite{Freyd1964AbelianCategories}.

If $\mathcal{S}$ is a Serre subcategory of $\mathcal{A}$ with the inclusion functor $i_{*} : \mathcal{S} \rightarrow \mathcal{A}$,
then there exists an abelian category $\mathcal{A}/\mathcal{S}$ and an exact functor $j^{*}:\mathcal{A} \rightarrow \mathcal{A} / \mathcal{S}$, which is dense and whose kernel is $\mathcal{S}$. The category $\mathcal{A}/\mathcal{S}$ and the functor $j^{*}$ are characterized by the following universal property: For any exact functor $G:\mathcal{A} \rightarrow \mathcal{B}$ such that $\mathcal{S} \subset \mathrm{Ker}(G)$ there exists a unique functor $H:\mathcal{A}/\mathcal{S} \rightarrow \mathcal{B}$ such that
\begin{eqnarray}
\xymatrix@=.4cm{
\mathcal{A}\ar[rd]_{j^*} \ar[rr]^{G} & &  \mathcal{B}  & \\
  & \mathcal{A}/\mathcal{S} \ar@{.>}[ru]_{H}
   }
    \end{eqnarray} 
commutes, that is, such that there exists a natural equivalence of functors
$H\circ j^{*}\approx G$. If $H'$ is another such functor, then $H\approx  H'$ (\cite{Faith2012Algebra:I}, Corollary 15.9).

A Serre subcategory $\mathcal{S}$ of $\mathcal{A}$ is called \textsf{localizing} provided that the functor $j^{*}:\mathcal{A} \rightarrow \mathcal{A}/\mathcal{S}$ admits a right adjoint $j_{*}:\mathcal{A}/\mathcal{S} \rightarrow \mathcal{A}$ which is called \textsf{section functor}. In this case, $j_{*}$ is fully faithful (\cite{Geigle1991PerpendicularSheaves}, Prop. 2.2).
If $\mathcal{A}$ is a Grothendieck category, then a Serre subcategory $\mathcal{S}$ is localizing if and only if $\mathcal{S}$ is closed under arbitrary coproducts (\cite{Faith2012Algebra:I}, Theorem 15.11).
If $\mathcal{A}$ is a Grothendieck category and $\mathcal{S}$ a localizing subcategory, then the quotient category $\mathcal{A} / \mathcal{S}$ is again a Grothendieck category.

\begin{definition}

Let $\mathcal{A}$ be a Grothendieck category.
\begin{itemize}
\item (1) An object $X$ is \textsf{of finite type} if whenever there are subobjects $ X_i \subset X$
for $i \in I$ satisfying
\begin{equation*}
    \sum_{i \in I} X_i =X,
\end{equation*}
then there is already a finite subset $J \subset I$ such that
\begin{equation*}
    \sum_{i \in J} X_i =X.
\end{equation*}
The subcategory of objects  of finite type is denoted by $\mathsf{fg}(\mathcal{A})$.

\item (2) An object $X$ is \textsf{finitely presented} if it is of finite type and if for any epimorphism 
$p:Y\rightarrow X$
 where 
$Y$
 is of finite type, it follows that 
$\mathrm{ker} (p)$
 is also of finite type. 
The subcategory of objects  of finite type is denoted by $\mathsf{fp}(\mathcal{A})$.
 
\item (3) An object  $X$ is \textsf{coherent} if it is of finite type and for any morphism $f
:
Y
\rightarrow
X$
 where 
$Y$
 is of finite type 
$\mathrm{ker} (f)$
 is of finite type.
 The subcategory of coherent objects is denoted by $\mathsf{coh}(\mathcal{A})$.
\end{itemize}
\end{definition}

\begin{definition}
\label{definition: locally coherent Grothendieck category}
Let $\mathcal{A}$ be a Grothendieck category.
\begin{itemize}
    \item (1) $\mathcal{A}$ is of \textsf{finite type}  provided that $\mathcal{A}$ has a generating set of  objects of finite type.
    \item (2) $\mathcal{A}$ is \textsf{locally finitely presented} provided that $\mathcal{A}$ has a generating set of finitely presented objects.
    \item (3) $\mathcal{A}$ is \textsf{locally coherent}  provided that every object of $\mathcal{A}$ is a direct limit of coherent objects.
\end{itemize}
\end{definition}

\begin{theorem}(\cite{Herzog1997TheCategory}, Theorem 1.6)
Let $\mathcal{A}$ be a locally finitely presented Grothendieck category, then
the following conditions are equivalent:
\begin{itemize}
    \item (1) $\mathcal{A}$ is locally coherent;
    \item (2) $\mathsf{fp}(\mathcal{A})$ is abelian;
    \item (3) $\mathsf{fp}(\mathcal{A})= \mathsf{coh}(\mathcal{A})$.
\end{itemize}
\end{theorem}

\begin{theorem}(\cite{Krause1997TheCategory}, Theorem 2.6)
\label{theorem: coherent-equivalence}
Let $\mathcal{A}$ be a locally coherent category and suppose that $\mathcal{C}$ is a Serre subcategory of $\mathcal{A}$. Then the following are equivalent:
\begin{itemize}
    \item (1) The inclusion functor $\mathcal{C}\rightarrow \mathcal{A}$ admits a right adjoint $t: \mathcal{A}\rightarrow \mathcal{C}$.
    \item (2) The quotient functor $q:\mathcal{A} \rightarrow \mathcal{A}/\mathcal{C}$ admits a right adjoint $s: \mathcal{A}/\mathcal{C}\rightarrow \mathcal{A}$.
\end{itemize}
If the above conditions (1)--(2) are satisfied, then the following are equivalent:
\begin{itemize}
    \item (3) $t$ commutes with direct limits;
    \item (4) $s$ commutes with direct limits.
\end{itemize}
If the above conditions (1)--(4) are satisfied, then $\mathcal{C}$ and $\mathcal{A}/\mathcal{C}$ are locally coherent and the following diagram of functors commutes where the unlabeled functors are inclusions.
\begin{eqnarray*}
 \xymatrix{
 \mathsf{fp}(\mathcal{C})  \ar[r] \ar[d]  &  \mathsf{fp}(\mathcal{A})  \ar[r]_{\mathsf{fp}(q)} \ar[d] & \mathsf{fp}(\mathcal{A})/\mathsf{fp}(\mathcal{C})  \ar[d]  \\
\mathcal{C}  \ar[r]  & \mathcal{A}  \ar[r]_{q} & \mathcal{A}/\mathcal{C}
   }\\
    \end{eqnarray*}
Moreover, there is a unique functor $f: \mathsf{fp}(\mathcal{A})/\mathsf{fp}(\mathcal{C})\rightarrow \mathsf{fp}(\mathcal{A}/\mathcal{C})$ such that $\mathsf{fp}(q)=f\circ p$, where $p:\mathsf{fp}(\mathcal{A})\rightarrow \mathsf{fp}(\mathcal{A})/\mathsf{fp}(\mathcal{C})$ denotes the quotient functor. The functor $f$ is an equivalence.
\end{theorem}

\begin{proposition}
\label{prop:qgr=fp Qgr}
Let $A$ be a 
finitely generated 
right graded coherent algebra.  Then the inclusion functor $\mathsf{gr} (A) \to \mathsf{Gr} (A)$
induces an equivalence of categories $\mathsf{qgr} A\equiv \mathsf{fp} (\Qgr A)$.
\end{proposition}

\begin{proof}
By definition, $\mathsf{Gr} (A)$  is a locally coherent category with $\mathsf{fp} (\mathsf{Gr} (A))  = \mathsf{gr} (A)$. Moreover, $\mathsf{Tors} (A)$ is its Serre subcategory and $\mathsf{fp}( \mathsf{Tors} (A))= \mathsf{tors} (A)$. Since $\mathsf{Tors} (A)$ is a localizing subcategory,  the condition (1) of Theorem \ref{theorem: coherent-equivalence} is satisfied. Note that the right adjoint to the inclusion functor $\mathsf{Tors} (A) \to \Gr A$ is the functor $t: \mathsf{Gr} (A)\rightarrow \mathsf{Tors} (A)$
 which sends a module $M$ to its maximal torsion submodule $\Tors M$.

Let us show that this functor $t$
commutes with direct limits. 
Let $(\Lambda,\leq)$ be a directed partially ordered set, $\left(
(M_{\lambda})_{\lambda \in \Lambda}, \{ f_{\lambda \mu}\}_{\lambda \le \mu}
\right) $ be a direct system in $\mathsf{Gr} (A)$ 
%a family of objects 
indexed by $\Lambda$, 
and $M= \varinjlim M_\lambda$. We identify $T = \varinjlim \Tors M_\lambda$ with a submodule 
of $M$. We need to prove that $T = \Tors M$.

%If $x \in T$ with $x = \varinjlim x_\lambda$ for 
%$x_\lambda \in \Tors M_\lambda$, then for each $\lambda$ 
%there exists  $n\ge 0$ such that $x_\lambda A_{\geq n} =0$. Then $x A_{\geq n} =0$
%and $x\in \Tors M$.

Since $T$ consists of torsion elements, we conclude that $T\subset \Tors M$. 
Reversely, for  $x\in \Tors M$ let $n\ge 0$ be such that  $x A_{\geq n} =0$.
The module $A_{\geq n}$ is finitely generated; let $Y$ be some finite generating set.
Then, for each $y\in Y$ there exists $\lambda(y)\in \Lambda$
and $x_y \in M_{\lambda(y)}$ such that $x_y$ represents $x$ and $x_y y =0$.
Since $Y$ is finite,  there exists $\mu \ge \max \{\lambda(y) | y\in Y\}$ and $x_\mu \in M_{\mu}$ such that $x_\mu = f_{\lambda(y), \mu} (x_y)$ for each $y\in Y$.
Hence, $x_\mu$ represents $x$ and $x_\mu A_{\ge n} = x_\mu F A =0$, so that 
$x_\mu \in \Tors M_{\mu}$. 
%If $x = \varinjlim x_\lambda$ for 
%$x_\lambda \in M_\lambda$, then for each $y\in Y$ there exists $\lambda = \lambda(y)\in \Lambda$ such that $x_\lambda y = 0$.
%Let $\mu$ be an upper bound of all $\lambda(y)$ for $y\in Y$. Then $x_{\mu} A_{\geq n} =0$, 
%so that $x_\lambda \in \Tors M_\lambda$ for all $\lambda \ge \mu$. 
It follows that $x \in T$. 
Thus, $T = \Tors M$.

Hence the functor $t$
commutes with direct limits, and the condition (3) of Theorem \ref{theorem: coherent-equivalence} is satisfied. Therefore
\begin{equation*}
\mathsf{qgr} (A) \equiv \mathsf{fp}(\mathsf{Gr} (A)/\mathsf{Tors} (A)) \ (= \mathsf{fp}(\mathsf{Qgr}(A))). \ \ \square
\end{equation*}
\end{proof}

\section{Ufnarovski graph and Holdaway--Smith map}

\label{sec:Ufnarovski_n_Holdaway--Smith}

\subsection{The Ufnarovski graph}
\label{subs:The Ufnarovski graph}

In \cite{Ufnarovskii1982AWords}, Ufnarovskii associates to any monomial algebra $A$ a directed graph $Q_A$ which has the same growth.
The Ufnarovskii graph is  defined as follows.

Let  $A={k\Gamma}/{(F)}$ be a graded monomial algebra of the form (\ref{form:monomial algebra}). 
Words not in $(F)$ are called \textsf{legal}. The set of legal words is denoted ${L}$. We write ${L}_n $ for the set of legal words of length $n$.
Words in $F$ are said to be \textsf{forbidden}. Let $l + 1$ be the maximum length of a forbidden word, i.e.,
\begin{equation*}
l + 1 := \mathrm{max} \ \{t | F\cap (k\Gamma)_{t} \not= \emptyset \}.
\end{equation*}

The \textsf{Ufnarovskii graph}  of $A$ is denoted ${Q}_A=({Q}_0,{Q}_1,s,t)$, where $Q_0$ is the set of vertices, $Q_1$ the set of arrows, and $s, t : Q_1 \rightarrow Q_0$ are maps which assign to each arrow $w$ its starting vertex $s(w)$ and its terminating vertex $t(w)$. 
Ufnarovskii graph is defined as follows. 
\begin{eqnarray*}
& {Q}_0 := {L}_{l}.\\
&{Q}_1 := {L}_{l+1}.\\
&s(w) := \text{the  unique word in} \  {L}_l \  \text{such  that}\  w\in s(w)L_1.\\
&t(w) := \text{the  unique word in} \  {L}_l \  \text{such  that}\  w\in L_1t(w).
\end{eqnarray*}

\begin{example}\label{example}
Let $A$ be the monomial algebra 
\begin{equation*}
A = \frac{k\langle x, y, z \rangle}{(x^2, yx, zy, xz, z^2, y^4)}.
\end{equation*}
The sets of legal words of length 3 and 4 are 
\begin{eqnarray*}
&Q_0 = \{ xy^2, xyz, y^2z, yzx, zxy, y^3 \}; \\
& Q_1 = \{ xy^2z, xyzx, y^2zx, yzxy, zxy^2, zxyz, y^3z, xy^3 \}. 
\end{eqnarray*}
Hence, the Ufnarovskii graph $Q_A$ is given by
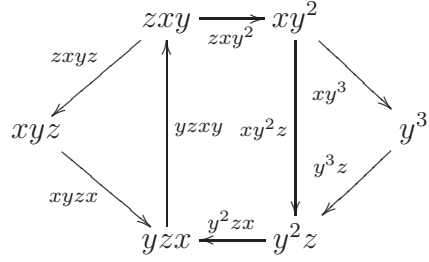
\begin{eqnarray*}
 \xymatrix{
    &  zxy  \ar[r]_{zxy^2} \ar[ld]_{zxyz} &  xy^2   \ar[dd]_{xy^2z} \ar[rd]_{xy^3}&  \\
     xyz  \ar[rd]_{xyzx}  &  &   & {y^3 } \ar[ld]_{y^3z} \\
    & yzx  \ar[uu]_{yzxy} & y^2z\ar[l]_{y^2zx}  &
   }\\
    \end{eqnarray*}

\end{example}

We label arrows in $Q_A$ by elements in $L_1$. The label attached to an arrow $w$ is the first letter of $w$. For example, the label attached to $zxyy$ is $z$.

\begin{example}
The labeling for the Ufnarovskii graph in example \ref{example} is
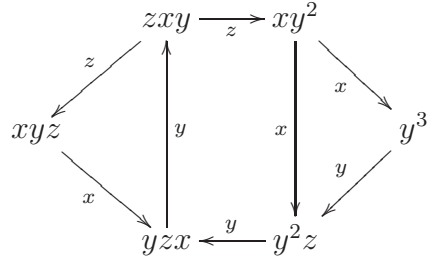
\begin{eqnarray*}
 \xymatrix{
    &  zxy  \ar[r]_{z} \ar[ld]_{z} &  xy^2   \ar[dd]_{x} \ar[rd]_{x}&  \\
     xyz  \ar[rd]_{x}  &  &   & {y^3 } \ar[ld]_{y} \\
    & yzx \ar[uu]_{y} & y^2z\ar[l]_{y}   &
   }\\
    \end{eqnarray*}
\end{example}

\begin{proposition}
%\label{Main theorem-3}
\label{prop:Main theorem-3}
\begin{equation*}
\mathsf{h}_{alg}(A) =  \mathsf{h}_{alg}(kQ_A).
\end{equation*}

\end{proposition}

\begin{proof}
If $m>l$, then each word $x_1x_2\cdots x_m$ not in $(F)$  is uniquely associated with a path labeled $x_1x_2\cdots x_{m-l}$ of length $m-l$. This correspondence is bijective \cite[Lemma 3.1]{Holdaway2012AnQuivers}, \cite{Ufnarovskii1982AWords}.
Hence 
\begin{equation*}
\mathsf{dim}_k A_m = \mathsf{dim}_k (kQ_A)_{m-l},
\end{equation*}
and we have 
\begin{equation*}
\mathsf{h}_{alg}(A) = \overline{\lim}_{m\rightarrow +\infty} \sqrt[m]{\mathsf{dim}_k A_m} = \overline{\lim}_{m\rightarrow +\infty} \sqrt[m+l]{\mathsf{dim}_k (k Q_A)_m} =  \mathsf{h}_{alg}(kQ_A). 
\end{equation*}
\end{proof}

\subsection{The Holdaway--Smith homomorphism}
Holdaway and Smith~\cite{Holdaway2012AnQuivers} have defined an algebra  homomorphism $f: A \to kQ_A$, this homomorphism were denoted by $\bar f$. 
It is defined on the generators of $A$ as follows.

Let $f:k\Gamma \rightarrow kQ_A$ be the unique algebra homomorphism such that for all $x \in L_1$,
\begin{equation*}
f(x) = \left\{
             \begin{array}{lr}
             \text{the  sum  of  all  arrows  labeled} \  x, &  \\
             0, \text{if there are no arrows labeled} \ x .&
             \end{array}
\right.
\end{equation*}
Then the homomorphism $f:k\Gamma \rightarrow kQ_A$ induces a homomorphism of graded algebras $\bar f :A\rightarrow kQ_A$ (\cite{Holdaway2012AnQuivers}, Proposition 3.3).

\begin{theorem}(\cite{Holdaway2012AnQuivers}, Theorem 1.1)
\label{th:Holdaway-Smith}
Let $A$ be a graded monomial algebra of the form (\ref{form:monomial algebra}). 
Let $Q_A$ be its Ufnarovskii graph and view $k$ as a left $A$-module through the homomorphism $\bar f : A \rightarrow kQ_A$.
Then $-  \otimes_A  kQ_A$ induces an equivalence of categories 
$\mathsf{Qgr} (A) \equiv \mathsf{Qgr} (kQ_A)$.
\end{theorem}

\begin{theorem}
\label{theorem: The Holdaway--Smith homomorphism induces an equivalence of qgr}
Let $A$ be a  graded monomial algebra of the form (\ref{form:monomial algebra}). 
Let $Q_A$ be its Ufnarovskii graph and view $k$ as a left $A$-module through the homomorphism $\bar f : A \rightarrow kQ_A$.
Then the functor $-  \otimes_A  kQ_A$ induces an equivalence of categories 
$\mathsf{qgr} (A) \equiv \mathsf{qgr} (kQ_A)$.
\end{theorem}
\begin{proof}
Since $A$ and $kQ_A$ are right coherent rings by Corollary~\ref{corollary: coherent algebras}, we can apply 
Proposition~\ref{prop:qgr=fp Qgr} and conclude that 
\begin{equation*}
\mathsf{qgr} (A) \equiv \mathsf{fp}(\mathsf{Gr} (A)/\mathsf{Tors} (A)) \ (= \mathsf{fp}(\mathsf{Qgr}(A)))
\end{equation*}
and
\begin{equation*}
\mathsf{qgr} (kQ_A) \equiv \mathsf{fp}(\mathsf{Gr} (kQ_A)/\mathsf{Tors} (kQ_A)) \ (= \mathsf{fp}(\mathsf{Qgr}( kQ_A))).
\end{equation*}
The functor $-  \otimes_A  kQ_A$ induces an equivalence of categories 
$\mathsf{Qgr} (A) \equiv \mathsf{Qgr} (kQ_A)$ by Theorem~\ref{th:Holdaway-Smith}, hence the functor
 $-  \otimes_A  kQ_A$ induces an equivalence of categories 
$\mathsf{fp}(\mathsf{Qgr}(A)) \equiv \mathsf{fp}(\mathsf{Qgr}(kQ_A))$, i.e., the functor $-  \otimes_A  kQ_A$ induces an equivalence of categories 
\begin{equation*}
\mathsf{qgr} (kQ_A) \equiv \mathsf{qgr} (A).    \ \  \square
\end{equation*}
\end{proof}

\section{Relationships between algebric, topological and category-theoretical entropies}
\label{sec:main_theorem}

Let $A$ be a graded monomial algebra of the form (\ref{form:monomial algebra}) over a field $k$, $kQ_A$ be the Ufnarovskii graph of $A$. 
$\pi: \mathsf{Gr}(kQ_A) \rightarrow \mathsf{Qgr}(kQ_A)$ be the projection functor.

\begin{notation}
\begin{itemize}
    \item (1) We write $e_i$ for the trivial path at vertex $i$. The indecomposable projective left $kQ_A$-modules are $P_1 = (kQ_A)_{e_1}, \cdots, P_n = (kQ_A)_{e_n}$.
    \item (2) We define $\mathcal O := \pi(kQ_A)$.
    \item (3) We write $\mathcal P_i = \pi(P_i)$ for the images of the indecomposable projectives in $\mathsf{Qgr}(kQ_A)$.
    \item (4) Given a graded $A$-module $M$,  we denote by $\mathsf{S}(M)$ the graded $A$-module with $\mathsf{S}(M)_d=M_{d+1}$. This is called the \textsf{Serre twist} of $M$. 
\end{itemize}
\end{notation}

The Serre twist  $\mathsf{S}$ induces an autoequivalence of $\mathsf{qgr} (kQ_A) $, which will be denoted by the same letter $\mathsf{S}$. 
%Then the pair $(\mathsf{qgr} (kQ_A), \mathsf{S})$ forms a triangulated category~\cite[1.5]{Smith2012CategoryQuivers}. 
Since the equivalence
of the categories $\mathsf{qgr}$ in Theorem~\ref{theorem: The Holdaway--Smith homomorphism induces an equivalence of qgr} commutes with $\mathsf{S}$, we get

\begin{proposition}
\label{Proposition: semi-simple}
Let $\mathcal O$ be as above. 
\begin{itemize}
     \item (1) (\cite{Smith2012CategoryQuivers}, Lemma 3.3) $\mathcal O$ is a projective object in $\mathsf{Qgr}(kQ_A)$.
    \item (2) (\cite{Smith2012CategoryQuivers}, Proposition 3.6) $\mathcal O$ is a generator in $\mathsf{qgr}(kQ_A)$.
    \item (3) (\cite{Smith2012CategoryQuivers}, Proposition 3.2) $\mathsf{qgr}(kQ_A)$ is a semi-simple category.
\end{itemize}
\end{proposition}

%\begin{corollary}
%\label{cor: qgr A is triangular}
%The pair $(\mathsf{qgr} (A), \mathsf{S})$ is a triangulated category in which %each distinguished triangle splits. The functor the functor $-  \otimes_A  %kQ_A$ induces an equivalence of triangulated categories 
%$(\mathsf{qgr} (A), \mathsf{S}) \equiv (\mathsf{qgr} (kQ_A), \mathsf{S})$.
%\end{corollary}

If $\mathcal E$ is an object of an abelian category $\mathsf{A}$, let $\mathcal E^{\bullet}$ denote the complex
\begin{equation*}
\mathcal E^{\bullet}= \cdots \xrightarrow{0} 0 \xrightarrow{0} \mathcal E \xrightarrow{0} 0 \xrightarrow{0} 0 \xrightarrow{0} \cdots 
\end{equation*}
which is $\mathcal E$ in degree zero, and $0$ in other degrees. Then $\mathsf{H}^0(\mathcal E^{\bullet}) = \mathcal E$, so $\mathcal E \mapsto \mathcal E^{\bullet}$ is a fully faithful embedding of $\mathsf{A}$ into $\mathbf{D}^b(\mathsf{A})$, with left inverse given by the functor $\mathsf{H}^0$. Usually we just identify $\mathcal E$ with $\mathcal E^{\bullet}$ and regard $\mathsf{A}$ as a full subcategory of $\mathbf{D}^b(\mathsf{A})$.

%\begin{definition}
%We say an $\mathsf A$-complex
%\begin{equation*}
%\mathcal{E}^{\bullet}= \cdots \xrightarrow{d^{-2}} \mathcal{E}^{-1} \xrightarrow{d^{-1}} \mathcal{E}^{0} \xrightarrow{d^0} \mathcal{E}^{1} \xrightarrow{d^1} \mathcal{E}^{2} \xrightarrow{d^2} \ldots\xrightarrow{d^{j-1}} \mathcal{E}^{j} \xrightarrow{d^{j}}  \cdots.
%\end{equation*}
%is \textsf{minimal} if $d^{j}=0$ for all $j\in \mathbb Z$, i.e., $\mathsf{H}^{j}(\mathcal{E}^{\bullet}) = E^j$ for all $j\in \mathbb Z$. 
%\end{definition}
%\begin{remark}
%If $\mathsf A$ is a semi-simple category, then every complex is quasi-isomorphic to a minimal complex. 
%\end{remark}

\begin{lemma}
\label{lemma: direct summand}
Let $\mathsf A$ be a semi-simple abelian category with a generator $\mathcal O$ and $\mathbf{D}^b(\mathsf A)$ be its derived category, then 
\begin{itemize}
\item (1) $\mathcal O$ is also a generator in $\mathbf{D}^b(\mathsf A)$.
\item (2) If there is a distinguished triangle
\begin{equation*}
    A^{\bullet} \rightarrow B^{\bullet} \rightarrow \mathcal C^{\bullet} \rightarrow  A^{\bullet} [1]
\end{equation*}
in $\mathbf{D}^b(\mathsf A)$, 
then 
$\mathsf{H}^n(B^{\bullet})$ is a direct summand of $\mathsf{H}^n(A^{\bullet})\oplus \mathsf{H}^n(C^{\bullet})$ for each $n \in \mathbb{Z}$. 

\item (3) If there is a tower of distinguished triangles
\begin{eqnarray*}
 \xymatrix@=.4cm{
  E^{\bullet}_0 \ar[rr]^{} & &  E^{\bullet}_1 \ar[rr]^{} \ar[ld]^{} & &E^{\bullet}_2 \ar[r]^{} \ar[ld]^{}&\cdots \ar[r]^{}& E^{\bullet}_{p-1}\ar[rr]^{} & &E^{\bullet}_p \ar[ld]^{}  \\
  &  \mathcal{O}[n_1] \ar@{-->}[lu]_{} & & \mathcal{O}[n_2]  \ar@{-->}[lu]_{}& & \cdots  & & \mathcal{O}[n_p] \ar@{-->}[lu]_{}
   }
    \end{eqnarray*} 
in $\mathbf{D}^b(\mathsf A)$ with $E^{\bullet}_0 = 0$, $p\geq 0$, and $n_i\in \mathbb{Z}$. Then $\mathsf{H}^n(E^{\bullet}_p)$ is a direct summand of $\mathcal O^{\oplus s}$, where $s$ is the number of $\mathcal O[-n]$. 
\end{itemize}
\end{lemma}
\begin{proof}
(1) Since $\mathsf A$ is a semi-simple category, 
every object 
\begin{equation*}
M^{\bullet}= \cdots \xrightarrow{d^{-2}} M^{-1} \xrightarrow{d^{-1}} M^{0} \xrightarrow{d^0} M^{1} \xrightarrow{d^1} M^{2} \xrightarrow{d^2} \ldots\xrightarrow{d^{j-1}} M^{j} \xrightarrow{d^{j}}  \cdots
\end{equation*}
in $\mathbf{D}^b(\mathsf A)$ is isomorphic to the direct sum of its cohomologies 
\begin{equation*}
\cdots \xrightarrow{0} \mathsf{H}^{-1} (M^{\bullet}) \xrightarrow{0} \mathsf{H}^{0} (M^{\bullet}) \xrightarrow{0} \mathsf{H}^{1} (M^{\bullet}) \xrightarrow{0} \mathsf{H}^{2} (M^{\bullet}) \xrightarrow{0} \ldots\xrightarrow{0} \mathsf{H}^{j} (M^{\bullet}) \xrightarrow{0}  \cdots
\end{equation*}
(see, for example, \cite[2.5]{Keller2010DerivedTilting}). 
Hence $\mathcal O$ is a generator in $\mathbf{D}^b (\mathsf A)$. 

(2) There is a long exact sequence
\begin{equation*}
 \cdots \rightarrow \mathsf{H}^ n ( A^{\bullet}) \rightarrow \mathsf{H}^ n (B^{\bullet}) \rightarrow \mathsf{H}^ n(C^{\bullet}) \xrightarrow{\delta} \mathsf{H}^ {n+1} ( A^{\bullet}) \rightarrow \mathsf{H}^ {n+1} ( B^{\bullet}) \rightarrow \mathsf{H}^ {n+1} ( C^{\bullet}) \rightarrow \cdots
\end{equation*}
in $\mathsf A$, hence $\mathsf{H}^n(B^{\bullet})$ is a direct summand of $\mathsf{H}^n(A^{\bullet})\oplus \mathsf{H}^n(C^{\bullet})$ for each $n$. 

(3) Since
\begin{equation*}
    \mathsf{H}^{j}(\mathcal{O}[t]) =\left\{
\begin{aligned}
0 \ \mbox{ for } j \ne -t,\\
\mathcal{O} \  \mbox{ for } j = -t,
\end{aligned}
\right.
\end{equation*}
we have $\mathsf{H}^n(E^{\bullet}_p)$ is a direct summand of $\mathcal O^{\oplus s}$, where $s$ is the number of $\mathcal O[-n]$ by (2). 
$\square$
\end{proof}

For a generator $\mathcal{O}$ and an object $X$ of a semi-simple abelian category $\mathsf{A}$, let us denote by 
$
\rk_{\mathcal{O}}(X) 
$ the minimal number of copies of $\mathcal{O}$ that cover $X$, that is, the minimal $s$ such that there is an epimorphism $\mathcal O^{\oplus s} \to X \to 0$ (equivalently, $X$ is a direct summand of $\mathcal O^{\oplus s}$).  

\begin{lemma}
\label{lemma: key lemma for derived}
\label{lemma: rk=delta}
Let  $\mathcal{O}$ be a generator and  $X$ be an object of a semi-simple abelian category $\mathsf{A}$.
Then 
\begin{equation*}
    \delta_t(\mathcal{O}, X) = \rk_{\mathcal{O}}(X),
\end{equation*}
where we identify the object $X$ of  $\mathsf{A}$
with the object $X^{\bullet}$ of $\mathbf{D}^b(\mathsf{A})$.
\end{lemma}
\begin{proof}
Suppose that there is an object $X'\in \mathbf{D}^b(\mathsf{A})$ and a tower of distinguished triangles
\begin{eqnarray*}
\xymatrix@=.4cm{
  X_0 \ar[rr]^{} & &  X_1 \ar[rr]^{} \ar[ld]^{} & &X_2 \ar[r]^{} \ar[ld]^{}&\cdots \ar[r]^{}& X_{p-1}\ar[rr]^{} & &X_p\cong X\oplus X'  \ar[ld]^{}  \\
  &  \mathcal{O}[n_1] \ar@{-->}[lu]_{} & & \mathcal{O}[n_2]  \ar@{-->}[lu]_{}& & \cdots  & & \mathcal{O}[n_p] \ar@{-->}[lu]_{}
   }
    \end{eqnarray*} 
with $X_0 = 0$, $p\geq 0$, and $n_i\in \mathbb{Z}$.
Then by Lemma \ref{lemma: direct summand} (3), 
$X$ is a direct summand of $\mathcal O^{\oplus s}$, where $s$ is the number of $\mathcal O[0]$, hence
the complexity  $\delta_t(\mathcal{O}, \mathcal E)\geq s \geq \rk_{\mathcal{O}}(X)$.

Now if $\rk_{\mathcal{O}}(X)=e$, then there is an epimorphism $\mathcal O^{\oplus e} \to X \to 0$, i.e., $X$ is a direct summand of $\mathcal O^{\oplus e}$.  
Hence we have a tower of distinguished triangles
\begin{eqnarray*}
\xymatrix@=.4cm{
  0 \ar[rr]^{} & &  \mathcal{O} \ar[rr]^{} \ar[ld]^{} & &\mathcal{O}^{\oplus 2} \ar[r]^{} \ar[ld]^{}&\cdots \ar[r]^{}& \mathcal{O}^{\oplus e-1}\ar[rr]^{} & &\mathcal{O}^{\oplus e}\cong X\oplus X'  \ar[ld]^{}  \\
  &  \mathcal{O}[0] \ar@{-->}[lu]_{} & & \mathcal{O}[0]  \ar@{-->}[lu]_{}& & \cdots  & & \mathcal{O}[0] \ar@{-->}[lu]_{}
   }
    \end{eqnarray*} 
Hence
\begin{equation*}
    \delta_t(\mathcal{O}, X) \leq e = \rk_{\mathcal{O}}(X),
\end{equation*}
and the lemma is proved.
$\square$
\end{proof}

\begin{theorem} 
\label{Main theorem-2}
\label{th: H_t = log h_alg}
In the notation above, 
\begin{equation*}
  \mathsf{h}_{t}(\mathbf{D}^b (\mathsf{qgr}(kQ_A)), \mathsf S) = 
 % \mathrm{lim}_{m\rightarrow +\infty} \frac{1}{m} \log N_m = 
  \log  \mathsf{h}_{alg}(kQ_A).
\end{equation*}
\end{theorem}

\begin{proof}
Denote $B = kQ_A$.
Let $U$ be the set of all vertices $u$ of $Q=Q_A$ that have no outgoing infinite paths, that is, 
$\dim u B <\infty$, and let $Q'$ be the full subgraph of $Q_A$ spanned by $Q_0\setminus U$. 
It follows from~\cite[Prop.~4.2]{Smith2012CategoryQuivers} that $ \mathsf{qgr}(kQ_A))\simeq \mathsf{qgr}(kQ'))$;
the equivalence, being induced by a homogeneous algebra map, is compatible with the degree shift. Then
$\mathsf{h}_{t}(\mathbf{D}^b (\mathsf{qgr}(kQ_A)), \mathsf S) = \mathsf{h}_{t}(\mathbf{D}^b (\mathsf{qgr}(kQ')), \mathsf S)$.
On the other hand, the characteristic polynomials $\chi (\lambda) $ and $\chi' (\lambda)$ of the adjacency matrices
$A_Q$ and  $A_{Q'}$ are connected by the equality $\chi (\lambda) = \chi' (\lambda) \cdot (-\lambda)^{|U|}$, so that
$$
\mathsf{h}_{alg}(kQ_A ) = \rho (A) = \rho (A') = \mathsf{h}_{alg}(kQ').  
$$
Thus, it is sufficient to prove the theorem in the case of quiver $Q = Q'$ without sinks. Then the module 
$B_B$ and all projective $B$-modules  are torsion free. 

%Then $\Tors B_B = BUB$  and $B' = kQ' \simeq B_B/Tors B_B$.

 Let $m\ge 0$ and let $s = \delta_t(\mathcal{O}, \mathcal O(m))$.
By Lemma~\ref{lemma: rk=delta}, we have $s = \rk_{\mathcal{O}} \mathcal O(m)$.
Then $\mathcal O(m)$ is a direct summand in $\mathcal O^{\oplus s}$ in $\mathsf{qgr}B$,
that is,
$\mathcal O^{\oplus s} \simeq
\mathcal O(m) \oplus \pi(M)$ for some 
torsion free module $M$.
%$M\in \mathsf{gr} B$.
It follows that for $n >>0$, we get the isomorphism of truncated modules in 
%$\mathsf{qgr}B$
$\mathsf{Gr}B$
%(see~\cite[Prop. 3.2 (2,3)]{Smith2012CategoryQuivers})
$$
  B_{\ge n}^{\oplus s} \simeq B_{\ge n+m}[m]   \oplus M_{\ge n}. 
$$
If $b_t$ denotes $\dim B_t$, we have 
$
    b_{m+n} \le s b_n,
$
so that $s \ge b_{m+n} /b_n $. 
It follows that there exists $q>0$ such that
$s \ge b_{m(s+1)+q} /b_{ms+q}
$
for all $s\ge 0$. Then for each $N>0$ we have $s^N \ge b_{mN+q}/b_q$. Thus, 
\begin{equation*}
s \ge 
\lim_{N\to \infty} \left( \frac{b_{mN+q}}{b_q} \right)^{1/N} = 
\lim_{N\to \infty} \left( {b_{mN+q}}^{1/(mN+q)} \right)^{(mN+q)/N} = 
\lim_{n\to \infty} \left( b_n \right)^{m/n} = 
%\lim_{n\to \infty} \frac{b_{m+n}}{b_n} =
%\lim_{n\to \infty} \frac{b_{n+1}}{b_n}\dots \frac{b_{n+m}}{b_{n+m-1}} = \left(  \lim_{n\to \infty} \frac{b_{n+1}}{b_n} \right)^m =
\mathsf{h}_{alg} (B)^m.
\end{equation*}
 % Should be replaced by:
%$s \ge 1/R$ be d'Alembert's criterion, where $R$ is the convergence radius of the series $\sum_n b_{mn}z^n$.
%Since $R = 1/\mathsf{h}_{alg}(B)^m$, we conclude that $s\ge \mathsf{h}_{alg} (B)^m$.

Therefore, 
\begin{equation*}
 \mathsf{h}_{t}(\mathbf{D}^b (\mathsf{qgr}(kQ_A)), \mathsf S) =
 \limsup_{m\rightarrow +\infty} \frac{1}{m} 
 \log
 %\log 
 \delta_t(\mathcal{O}, \mathcal{O}(m)) 
 \ge 
  \lim_{m\rightarrow +\infty} \frac{1}{m} \log \left( \mathsf{h}_{alg} (B)^m \right) =  \log  \mathsf{h}_{alg}(B).
\end{equation*}

To prove the opposite inequality, note that the condition $s = \rk_{\mathcal{O}} \mathcal O(m)$ implies that  there is no decomposition 
$$
\mathcal O(m) \simeq \bigoplus_{j} \mathcal{P}_{j}^{\oplus k_j}   
$$
with $k_j \le s-1$ for all $j$. Hence in the decomposition
$$
B_{\ge m}[m] \simeq \bigoplus_{j} P_{j}^{\oplus k_j} 
$$
of the projective $B$-module $B_{\ge m}[m]$ into the direct sum of the simple projectives 
%$P_j$ 
we have $k_j \ge s$ for some $j$.
It follows that $b_m \ge s$. Thus, 
$$
 \mathsf{h}_{t}(\mathbf{D}^b (\mathsf{qgr}(kQ_A)), \mathsf S) =
 \limsup_{m\rightarrow +\infty} \frac{1}{m} \log \delta_t(\mathcal{O}, \mathcal{O}(m)) 
 \le  \limsup_{m\rightarrow +\infty} \frac{1}{m} \log b_m 
 = \log  \mathsf{h}_{alg}(B).
$$
\end{proof}

\section{An example}
\label{section: An example}

Let $A$ be the monomial algebra 
\begin{equation*}
A = \frac{k\langle x, y, z \rangle}{(F)},
\end{equation*}
where $F=\{xz,yz\}$. We have $\mathsf{dim}_k(A_s)= 2^{s+1}-1$, therefore
\begin{equation*}
\mathsf{h}_{alg}(A) = \overline{\mathrm{lim}}_{s\rightarrow +\infty} \sqrt[s]{\mathsf{dim}_k(A_s)}= 2,
\end{equation*}
and
\begin{equation*}
  \mathsf{h}_{top}(X_F)=  \overline{\mathrm{lim}}_{s\rightarrow +\infty} \frac{1}{s} \log|\mathsf{dim}_k(A_s)|=1.
\end{equation*}

The sets of legal words of length $1$ and $2$ are
\begin{equation*}
Q_0 = \{  x, y, z  \} 
\end{equation*}
and
\begin{equation*}
Q_1 = \{ x^2, y^2, z^2,  xy , zx, zy, yx\}. 
\end{equation*}
Hence, the Ufnarovski graph $Q_A$ is given by
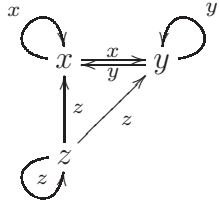
\begin{eqnarray*}
 \xymatrix{
  x \ar@(l,u)[]^{x} \ar[r]_{y} &  y \ar@< 2pt>[l]_{x}  \ar@(r,u)[]_{y}    \\
  z  \ar@(l,d)[]^{z} \ar[u]_{z} \ar[ru]_{z} & 
   }\\
    \end{eqnarray*}
Let $N_s^x , N_s^y ,N_s^z $ be the numbers of paths of length $n$ that start with $x,y,z$, we have
\begin{equation*}
N_s^x = 2^s, N_s^y = 2^s, N_s^z = 2^{s+1}-1,
\end{equation*}
and
\begin{equation*}
\mathsf{dim}_k(kQ_A)_{s} = N_s^x + N_s^y + N_s^z =  2^{s+2}-1. 
\end{equation*}
Therefore
\begin{equation*}
\mathsf{h}_{alg}(kQ_A) = \overline{\mathrm{lim}}_{s\rightarrow +\infty} \sqrt[s]{\mathsf{dim}_k(kQ_A)_{s}}= 2.
\end{equation*}

Let $\mathsf C =\mathbf{D}^b(\mathsf{qgr}(kQ_A))$,
and
let $\mathsf S : \mathsf C \rightarrow \mathsf C$ be
the Serre twist functor.
We write $e_x,e_y,e_z$ for the trivial paths at vertex $x,y,z$. 
The indecomposable projective left $kQ_A$-modules are
$P_x = (kQ_A)_{e_x}$,$P_y = (kQ_A)_{e_y}$,$P_z = (kQ_A)_{e_z}$.
We write $\mathcal{P}_x=\pi(P_x),\mathcal{P}_y=\pi(P_y),\mathcal{P}_z=\pi(P_z)$ for the images of the indecomposable
projectives in $QGr(kQ_A)$.
%By (\ref{formula: projective-1}) and (\ref{formula: projective-2}) we have
Then
\begin{equation*}
\mathcal{P}_x (1) \cong
\mathcal{P}_x \oplus \mathcal{P}_y \oplus \mathcal{P}_z; \ \mathcal{P}_y (1) \cong
\mathcal{P}_x \oplus \mathcal{P}_y \oplus \mathcal{P}_z; \ \mathcal{P}_z (1) \cong \mathcal{P}_z.
\end{equation*}
We can use induction and get 
\begin{equation*}
\mathcal{P}_x (s) \cong
\mathcal{P}_x^{\oplus 2^{s-1}} \oplus \mathcal{P}_y^{\oplus 2^{s-1}} \oplus \mathcal{P}_z^{\oplus 2^s-1}. 
\end{equation*} 
Then
$$
\mathcal{O}(s) \cong   (\mathcal{P}_x\oplus \mathcal P_y \oplus \mathcal P_z )(s-1) \cong \mathcal{P}_x(1) (s-1)
= \mathcal{P}_x (s)
\cong
\mathcal{P}_x^{\oplus 2^{s-1}} \oplus \mathcal{P}_y^{\oplus 2^{s-1}} \oplus \mathcal{P}_z^{\oplus 2^s-1}.
$$

Then the object $\mathcal{O}^{2^{s-1}}$ is a direct summand of $\mathcal{O}(s) = \mathsf S^s\mathcal{O} $;
in turn, the object $\mathcal{O}(s) $ is a direct summand in $\mathcal{O}^{2^s-1}$.
It follows that $2^{s-1} \le \rk_\mathcal{O} \mathsf {S}^s\mathcal{O} \le 2^s-1$. By
Lemma~\ref{lemma: rk=delta}, we have
%$\delta_t(\mathcal{O}, \mathsf S^s\mathcal{O})=\rk_\mathcal{O} \mathsf S^s\mathcal{O} = 2^{s+1}-1$ by 
\begin{equation*}
\mathsf{h}_t(\mathsf S)= 
\lim_{n\rightarrow \infty} \frac{1}{n} \log  \delta_t(\mathcal{O}, \mathsf S^n\mathcal{O})= \lim_{n\rightarrow \infty} \frac{1}{n} \log \rk_\mathcal{O} \mathsf S^n\mathcal{O}= 
%\log_2 (2) = 1.
\log 2.
\end{equation*}

 %\subsection*{Acknowledgement}

% An acknowledgements section is started with \verb"\ack" or \verb"\acks"
% for \textit{Acknowledgement} or \textit{Acknowledgements}, respectively. It
% must be placed just before the references (or before the appendix when applicable).

\ack The authors are grateful to anonymous reviewers for their comments and suggestions, which helped us to improve the text and avoid a number of inaccuracies. The work of D. Piontkovski has been supported by the grant of the Russian Science Foundation, RSF 22-21-00912. 

%\section*{References}
%\bibliography{additional_biblio,mybibfile}

%\usepackage[notes,backend=biber]{biblatex-chicago}
%\bibliography{sample}
%\bibliography{references}

% Compiled bibliography to submit to IMRN:

\end{document}